\documentclass[reqno,12pt]{amsart} 
\usepackage{vmargin}
\setpapersize{A4}
\usepackage{amsmath}

\usepackage{amsfonts}   
\usepackage{amssymb}      
\usepackage{eufrak}      
\usepackage{color}
\usepackage{amscd}      
\usepackage{amsthm}      
\usepackage{tikz,amsmath}
\usetikzlibrary{arrows}
\usepackage{tikz}
\usetikzlibrary{matrix,calc}

\usepackage{epsfig}      
\usepackage{amstext}      
\usepackage{graphicx}

\usepackage[all,line,dvips]{xy} 
\CompileMatrices %goer at xy-matricerne koerer hurtigere.

\newcommand{\Span}{\operatorname{Span}}
\newcommand{\Tr}{\operatorname{Tr}}

\newcommand{\T}{{\mathbb{T}}}

   \theoremstyle{plain}%default
   \newtheorem{thm}{Theorem}[section]
   \newtheorem{prop}[thm]{Proposition}
   \newtheorem{lemma}[thm]{Lemma}  
   \newtheorem{cor}[thm]{Corollary}
   \theoremstyle{definition}
   
   \newtheorem{defn}[thm]{Definition}
   
   \theoremstyle{remark}
   
   \newtheorem{remark}[thm]{Remark}

\usepackage{mdframed,xcolor}

\definecolor{mybgcolor}{gray}{0.8}
\definecolor{myframecolor}{rgb}{.647,.129,.149}

\mdfdefinestyle{mystyle}{
  usetwoside=false,
  skipabove=0.6em plus 0.8em minus 0.2em,
  skipbelow=0.6em plus 0.8em minus 0.2em,
  innerleftmargin=.25em,
  innerrightmargin=0.25em,
  innertopmargin=0.25em,
  innerbottommargin=0.25em,
  leftmargin=-.75em,
  rightmargin=-0em,
  topline=false,
  rightline=false,
  bottomline=false,
  leftline=false,
  backgroundcolor=mybgcolor,
  splittopskip=0.75em,
  splitbottomskip=0.25em,
  innerleftmargin=0.5em,
  leftline=true,
  linecolor=myframecolor,
  linewidth=0.25em,
}

\newmdenv[style=mystyle]{important}

\usepackage{lipsum}

   \numberwithin{equation}{section}

        \date{\today}
\title[Finite digraphs and KMS states]{Finite digraphs and KMS states}
\author{Johannes Christensen and Klaus Thomsen}
\date{\today}

\email{matkt@math.au.dk; jollech90@hotmail.com}
\address{Department of Mathematics, Aarhus University, Ny Munkegade, 8000 Aarhus C, Denmark}

\begin{document}

\maketitle

\section{Introduction} 

In a recent paper by an Huef, Laca, Raeburn and Sims, \cite{aHLRS}, the authors
describe an algorithm by which it is possible to determine all the KMS
states of the gauge action on the $C^*$-algebra of a finite
graph. Their results cover also the gauge action on the Toeplitz
extension of the algebra and extend the result of Enomoto, Fujii and
Watatani, \cite{EFW}, which deals with strongly connected
graphs. Almost simultaneously with this work, Carlsen and Larsen obtained an abstract description of the
KMS states for some of the generalized gauge actions on the
$C^*$-algebra of a finite graph as well as its Toeplitz
extension. Their work build on and extend methods and results obtained
by Exel and Laca in
\cite{EL} and bring our knowledge about the KMS states of
the actions they consider to the point where the work on the gauge
action begins in \cite{aHLRS}. It is the purpose of the present paper to take the
steps from the abstract to the concrete which were taken by an Huef,
Laca, Raeburn and Sims, but now for all the generalized gauge actions.

The point of departure for our work are results of the second author
from \cite{Th3} from which it follows that the relevant results of
Carlsen and Larsen from \cite{CL} remain valid for all generalized
gauge actions, provided attention is restricted to the KMS states that
are gauge invariant; a condition which is automatically satisfied for
the actions considered by Carlsen and Larsen. What we do first is to
develop the approach from \cite{aHLRS} to make it applicable to
generalized gauge actions. In this way we obtain a description of the
gauge invariant KMS states for all generalized gauge actions. The main
input for this is a generalization of the Perron-Frobenius theory for
positive matrices which was obtained by Victory, \cite{Vi}. See also
\cite{Ta}. The theory handles arbitrary finite non-negative matrices and can also
be used to simplify some of the steps in \cite{aHLRS}. We give here a new
proof of the relevant results from \cite{Vi} and \cite{Ta} by using
ideas from \cite{aHLRS}.

In order to identify the KMS states that are not gauge invariant we
use results by Neshveyev, \cite{N}, in a form presented in
\cite{Th1}. By combining the result with our study of the gauge
invariant KMS states we obtain in Theorem \ref{MAIN} our main result which
describes the $\beta$-KMS states for all $\beta \in \mathbb R
\backslash \{0\}$ and for
an arbitrary
generalized gauge action on the $C^*$-algebra of a finite graph. As
with the gauge action, \cite{aHLRS}, it is a sub-collection of the components and
the sinks in the graph that parametrize the extremal KMS states,
although in general some of the components, corresponding to a loop
without exits, may contribute a family of extremal KMS states
parametrized by a circle. Which components and sinks play a
role depends very much on the action, as we illustrate by examples.

It is
intrinsic to our approach that the case $\beta = 0$, where the KMS
states are the trace states, must be handled separately as we do in
Section \ref{traces}. For completeness we describe also in a final section the ground states for
the same actions. While there are no ground states for the gauge
action unless the graph has sinks, this is no longer the case for
generalized gauge actions and even for strongly connected graphs their
structure can be very rich. 

\smallskip
  
\emph{Acknowledgement.} The authors thank Astrid an Huef and Iain Raeburn for
          discussions on the subject of this paper.

%The description we obtain of the gauge invariant KMS states depends a
%result of Victory, \cite{Vi}, in a form given in \cite{Ta} where it is
%called 'The Frobenius-Victory theorem'. We present a direct proof
%which appears to much shorter than the proof in \cite{Vi}.

\section{Preparations}\label{sec1} Let $G$ be a finite directed graph with vertex
set $V$ and edge set $E$. The maps $r,s : E \to V$ associate to an
edge $e\in E$ its source vertex $s(e)\in V$ and range vertex $r(e) \in
V$. Thus the set of edges emitted from a vertex $v$ is the set
$s^{-1}(v)$ while $r^{-1}(v)$ is the set of edges terminating at
$v$. A \emph{sink} in $G$ is a vertex $v$ that does not emit an edge,
i.e. $s^{-1}(v) = \emptyset$. 

Formulated in terms of generators and relations the $C^*$-algebra
$C^*(G)$ of $G$ is the universal $C^*$-algebra
generated by a set $S_e, e \in E$, of partial isometries and a set
$P_v, v \in V$, of mutually orthogonal projections such that
\begin{equation}\label{CKrel}
\begin{split}
& 1) \ \ \ S_e^*S_e = P_{r(e)}  \ \forall e \in E, \ \text{and} \\
& 2) \ \ \ P_v = \sum_{e \in s^{-1}(v)} S_eS_e^* \ \text{for every vertex} \
v
  \in V \ \text{which is not a sink.}
\end{split}
\end{equation}
A finite path $\mu$ in $G$ is an element $\mu = e_1e_2\cdots e_n \in
E^n$ for some $n \in \mathbb N$ such that $r(e_i) = s(e_{i+1}), i =
1,2,\cdots,n-1$.   
%{\color{blue} For $X,Y \subseteq V$ and $n \in
                 %\mathbb{N}$ we define $XE^{n}Y=\{\alpha \in E^{n} \
                 %| \ s(\alpha) \in X, r(\alpha) \in Y \}$, and we use
                 %the abriviation $XEY$ for $XE^{1}Y$ and $XE^{*}Y$
                 %for $\bigcup_{n=1}^{\infty} XE^{n}Y$, and we write
                 %e.g. $XE$ for $XEV$, $E^{n}Y$ for $VE^{n}Y$ and so
                 %on.} 
For such a path we set
$$
S_{\mu} = S_{e_1}S_{e_{2}}  \cdots S_{e_{n-1}}S_{e_n} .
$$
The number $|\mu| = n$ is the length of the path. We
consider a vertex $v$ as a path $\nu$ of length $0$, and set $S_{\nu}
= P_v$ in this case. Let $P_f(G)$ denote the set of finite paths in
$G$. Then
\begin{equation}\label{span}
\mathcal A = \Span \left\{ S_{\mu}S_{\nu}^*: \ \mu, \nu \in P_f(G) \right\}
\end{equation}
is a dense $*$-subalgebra of $C^*(G)$. %In particular, a continuous
%linear functional on $C^*(G)$ is determined by the values it takes on
%elements of form
%$S_{\mu}S^*_{\nu}$. 

Let $F : E \to \mathbb R$ be a function. The universal property of $C^*(G)$
guarantees the existence of a one-parameter group $\alpha^F_t, t\in
\mathbb R$, on $C^*(G)$
such that
$$
\alpha^F_t(P_v) = P_v \ \forall v \in V, \ \ \ \text{and} \ \ \ \alpha^F_t(S_e)
= e^{iF(e)t} S_e \ \forall e \in E .
$$
For $\beta \in \mathbb R$ a \emph{$\beta$-KMS state} for $\alpha^F$ is
a state $\varphi$ on $C^*(G)$ such that
$$
\varphi(ab) = \varphi\left(b \alpha^F_{i\beta}(a)\right)
$$
for all $a,b \in \mathcal A$, cf. Definition 5.3.1 in \cite{BR}. When $F$ is constant $1$ the automorphism group
$\left\{\alpha^1_t\right\}$ is the so-called \emph{gauge
  action} and we study first the gauge invariant KMS states for $\alpha^F$, i.e the
KMS states $\varphi$ for $\alpha^F$ with the additional property that
$\varphi \circ \alpha^1_t= \varphi$ for all $t\in \mathbb R$. For this
purpose we use the following description of the gauge invariant KMS
states. It was obtained by Carlsen and Larsen in \cite{CL} when $F$ is
strictly positive (in which case all KMS states for $\alpha^F$ are
gauge-invariant). The general case follows from Theorem 2.8 in \cite{Th3}.

Let $B$ be a non-negative matrix over $V$ with the property that
$B_{vw} > 0$ iff there is an edge in $G$ from $v$ to $w$. A vector $\psi \in [0,\infty)^V$ is \emph{almost harmonic} for $B$
(or \emph{almost $B$-harmonic}) when
\begin{equation}\label{a77}
\sum_{w \in V} B_{vw} \psi_w = \psi_v \ 
\end{equation}
for every vertex $v \in V$ which is not a sink, 
and \emph{normalized} when $\sum_{v \in V} \psi_v =1$. When the
identity (\ref{a77}) holds for all $v \in V$ we say that $\psi$ is
\emph{harmonic} for $B$ (or \emph{$B$-harmonic}). Thus an almost
$B$-harmonic vector $\psi$ is $B$-harmonic if and only if $\psi_s
=0$ for every sink $s \in V$. For $\beta \in
\mathbb R$, consider the matrix $A(\beta) =
\left(A(\beta)_{vw}\right)$ over $V$ defined such that
$$
A(\beta)_{vw} \ = \ \sum_{e \in vEw}  e^{-\beta F(e)} ,
$$  
where $vEw = s^{-1}(v) \cap r^{-1}(w)$. For a finite path $\mu  = e_1e_2 \cdots e_n $ in $G$, set
$$
F(\mu) = F(e_1) + F(e_2) + \cdots + F(e_n) .
$$

\begin{lemma}\label{a2} (\cite{CL}, \cite{Th3}) For every normalized $A(\beta)$-almost
  harmonic vector $\psi$ there is a unique gauge invariant $\beta$-KMS state
  $\omega_{\psi}$ for $\alpha^F$ such that
\begin{equation}\label{a3}
\omega_{\psi}\left(S_{\mu} S_{\nu}^*\right) = \delta_{\mu, \nu}
e^{-\beta  F(\mu)} \psi_{r(\mu)}
\end{equation}
for every pair $\mu,\nu$ of finite paths in $G$. Furthermore, every
gauge invariant $\beta$-KMS state for $\alpha^F$ arises from a normalized $A(\beta)$-almost
  harmonic vector in this way.
\end{lemma}

%Thus there is an affine bijection between the gauge invariant $\beta$-KMS
%states for $\alpha^F$ and normalized $A(\beta)$-almost harmonic
%vectors. 
By Lemma \ref{a2} the study of the gauge invariant KMS states becomes a
study of normalized almost harmonic vectors for the family
$A(\beta),\beta \in \mathbb R$, of non-negative matrices over $V$.

\section{Almost harmonic vectors for a non-negative matrix}

Let $B$ be a non-negative matrix over $V$ with the property that
$B_{vw} > 0$ iff there is an edge in $G$ from $v$ to $w$. We seek to obtain a
description of the $B$-almost harmonic vectors. %Since the normalized
%$B$-almost harmonic vectors is a compact convex set this can be
%achieved by identifying the extremal points of this set. 

We shall need the following well-known lemma, cf. e.g. 6.43 in \cite{W}. 

\begin{lemma}\label{a6} (Riesz decomposition.) 
Let $\psi = (\psi_v)_{v \in V} \in [0,\infty[^V$ be a non-negative
vector such that 
\begin{equation*}\label{b2}
\sum_{w\in V} B_{vw} \psi_w \leq \psi_v
\end{equation*}
 for all $v\in V$. It follows that there are unique non-negative
 vectors $h, k \in [0,\infty[^V$ such that $h$ is $B$-harmonic
 and
\begin{equation}\label{b3}
\psi_v = h_v + \sum_{w \in V} \sum_{n=0}^{\infty} B^n_{vw}k_w 
\end{equation}
for all $v \in V$. The vector $k$ is given by
$$
k_v= \psi_v - \sum_{w \in V} B_{vw}\psi_w,
$$ 
while
$$
h_v = \lim_{n \to \infty} \sum_{w \in V} B^n_{vw}\psi_w .
$$ 
\end{lemma}

We say that a sink $s \in V$ is \emph{$B$-summable} when
$$
\sum_{n=0}^{\infty} B^n_{vs} < \infty 
$$
for all $v \in V$. For such a sink we define a vector $\phi^s \in
[0,\infty)^V$ by
$$
\phi^s_v =  \frac{\sum_{n=0}^{\infty} B^n_{vs}}{\sum_{w \in V}
  \sum_{n=0}^{\infty} B^n_{ws}} .
$$

\begin{lemma}\label{a5} $\phi^s$ in an extremal normalized $B$-almost harmonic vector.
\end{lemma}
\begin{proof} The only assertion which may not be straightforward to
  verify is that $\phi^s$ is extremal in the set of normalized
  $B$-almost harmonic vectors. To show this, consider a $B$-almost harmonic vector $\varphi$ with the property that
  $\varphi \leq \phi^s$. Since 
\begin{equation}
\begin{split}
&B^m_{vw}\varphi_w \leq B^m_{vw}\phi^s_w  \ \leq \ 
\frac{\sum_{n=m}^{\infty} B^n_{vs}}{\sum_{w \in V}
  \sum_{n=0}^{\infty}B^n_{ws}} \ \to \ 0 
\end{split}
\end{equation}
as $m \to \infty$, it follows that the harmonic part from the Riesz
decomposition of $\varphi$ is zero. Thus
$$
\varphi_v =  \sum_{w \in V} \sum_{n=0}^{\infty} B^n_{vw}k_w 
$$
where $k_v = \varphi_v - \sum_{w \in V} B_{vw}\varphi_w$. Note that
$k_v = 0$ when $v$ is not a sink since $\varphi$ is $B$-almost
harmonic, and that $k_{s'}= \varphi_{s'}$ for every sink $s'$. Note also that $\phi^s_{s'} = 0$ for every sink
$s'$ in $G$ other than $s$. Since $\varphi \leq \phi^s$ it follows
that the same is true for $\varphi$. Hence 
$$
\varphi_v =   \sum_{n=0}^{\infty} B^n_{vs}\varphi_s = t
\phi^s_v,
$$
where 
$$
t = \varphi_s \sum_{w \in V}
  \sum_{n=0}^{\infty}B^n_{ws} .
$$ 
\end{proof}

By combining Lemma \ref{a6} and Lemma \ref{a5} we obtain the following
\begin{prop}\label{b2} Let $\psi$ be a normalized $B$-almost
  harmonic vector. There are a unique (possibly empty) set $\mathcal
  S$ of  summable sinks in $G$, unique positive numbers $t_s \in
  ]0,1], s \in \mathcal S$, and a unique $B$-harmonic
  vector $h$ such that
$$
\psi = h + \sum_{s \in \mathcal S} t_s \phi^s .
$$

\end{prop}

We turn to a study of the $B$-harmonic vectors. For any pair of subsets $E,D \subseteq V$ we let
$B^{E,D}$ denote the $E \times D$-matrix obtained by restricting $B$
to $E \times D$, and we set $B^E = B^{E,E}$ for any subset $E
\subseteq V$.

Write $v \leadsto w$ between two vertexes $v,w$ when there is a finite path $\mu = e_1\cdots e_n$ in $G$ such
that $s(e_1) = v$ and $r(e_n) =w$, and $v \sim w$ when $v \leadsto
w$ and $w \leadsto v$. Then $\sim$ is an equivalence relation since we consider
a vertex $v$ as a finite path (of length $0$) from $v$ to $v$. A
\emph{component} $C$ in $G$ is an equivalence class in
$V/\sim$ such that $B^C \neq 0$. For any collection $F$ of vertexes in $G$ we define the
\emph{closure} of $F$ to be the set of
vertexes that 'talk' to an element of $F$, i.e. $v \in \overline{F}$ if
and only if there is a vertex $w \in F$ such that $v \leadsto w$. In
contrast the \emph{hereditary closure} of a set $F$ consists of the
vertexes $w \in V$ such that $v \leadsto w$ for some $v\in F$. The
hereditary closure will be denoted by $\widehat{F}$.

In the following we denote the spectral radius of a finite matrix $A$
by $\rho(A)$. A component $C$ in $G$ is \emph{$B$-harmonic} when
\begin{enumerate}
\item[a)] $\rho\left(B^C\right) = 1$ and
\item[b)] $\rho\left(B^{\overline{C} \backslash C}\right) < 1$ if
  $\overline{C} \backslash C \neq \emptyset$.
\end{enumerate}

This definition, as well as the proof of the following lemma, is inspired by Theorem 4.3 in \cite{aHLRS}.

\begin{lemma}\label{a7} Let $C$ be a $B$-harmonic component
  in $G$. There is a unique normalized $B$-harmonic vector $\phi^C$ such that
  $B^C \phi^C|_C = \phi^C|_C$ and
  $\phi^C_v \neq 0 \ \Leftrightarrow \ v \in \overline{C}$.
\end{lemma}
\begin{proof} Existence: Since $\rho\left(B^C\right) = 1$ it follows from
  Perron-Frobenius theory that there is a strictly positive vector $x^C \in [0,\infty)^C$ such that $B^C x^C = x^C$. Since
  $\rho\left(B^{\overline{C} \backslash C}\right) < 1$, the matrix
  $1^{\overline{C} \backslash C} - B^{\overline{C} \backslash C}$ is
  invertible and we set
$$
\phi^C = \left( 1^{\overline{C} \backslash C} - B^{\overline{C}
    \backslash C}\right)^{-1} B^{\overline{C} \backslash C, C}x^C \ +
\ x^C,
$$
which is a strictly positive vector in $[0,\infty)^{\overline{C}}$. For any pair of
vertexes $v,w\in \overline{C}\backslash C$,
$$
\limsup_n \left( B^n_{v,w} \right)^{\frac{1}{n}} \leq
\rho\left(B^{\overline{C} \backslash C}\right) < 1,
$$
and hence
$$
 \left( 1^{\overline{C} \backslash C} - B^{\overline{C}
    \backslash C}\right)^{-1} = \ \sum_{n=0}^{\infty} \left(B^{\overline{C}
    \backslash C}\right)^{n} .
$$  
Using this and that no vertex in
$C$ talks to a vertex in $\overline{C} \backslash C$, we find that
\begin{equation}\label{april1}
\begin{split}
&B^{\overline{C}}\phi^C = B^{\overline{C} \backslash C}\left( 1^{\overline{C} \backslash C} - B^{\overline{C}
    \backslash C}\right)^{-1} B^{\overline{C} \backslash C, C}x^C +
B^{\overline{C} \backslash C, C}x^C + B^Cx^C\\
& =  B^{\overline{C} \backslash C}\sum_{n=0}^{\infty} \left(B^{\overline{C}
    \backslash C}\right)^{n} B^{\overline{C} \backslash C, C}x^C + 
B^{\overline{C} \backslash C, C}x^C + x^C\\
& = \sum_{n=1}^{\infty} \left(B^{\overline{C}
    \backslash C}\right)^{n} B^{\overline{C} \backslash C, C}x^C + 
B^{\overline{C} \backslash C, C}x^C + x^C \\
& = \sum_{n=0}^{\infty} \left(B^{\overline{C}
    \backslash C}\right)^{n} B^{\overline{C} \backslash C, C}x^C +
x^C\\
& = \phi^C .
\end{split}
\end{equation}
Set $\phi^C_v = 0$ when $v \notin \overline{C}$ and normalize the
resulting vector in $[0,\infty)^V$. It follows from (\ref{april1})
that $\phi^C$ is $B$-harmonic. Since $\phi^C|_C$ is multiple of
$x^C$ by construction, it follows that $B^C\phi^C|_C = \phi^C|_C$.

Uniqueness: If $\psi$ is a normalized $B$-harmonic vector such that
$B^C\psi|_C = \psi|_C$ and $\psi_v \neq 0 \Leftrightarrow v
\in \overline{C}$, it follows from Perron-Frobenius theory that there is a
$\lambda> 0$ such that $\psi_v = \lambda \phi^C_v \ \forall v \in
C$. Then $\psi - \lambda \phi^C$ is vector supported in $\overline{C}
\backslash C$ such that $B^{\overline{C}\backslash C}(\psi - \lambda
\phi^C) = \psi - \lambda \phi^C$. Since $\rho\left(B^{\overline{C}
    \backslash C}\right) < 1$, it follows first that $\psi = \lambda
\phi^C$ and then that $\psi = \phi^C$ because both vectors are normalized.
\end{proof}

The following theorem is equivalent to the Frobenius-Victory theorem
stated as Theorem 2.7 in \cite{Ta}.

\begin{thm}\label{a8} Let $\psi \in [0,1]^V$ be a normalized
  $B$-harmonic vector. There is a unique collection $\mathcal
  C$ of
  $B$-harmonic components in $G$ and positive
  numbers $t_C \in ]0,1], C \in \mathcal C$,
 such that
\begin{equation}\label{a10}
\psi = \sum_{C \in \mathcal C} t_C\phi^C .
\end{equation}
\end{thm} 
\begin{proof} Set $\copyright = \left\{ v \in V: \ \psi_v >
    0\right\}$. Let $v \in \copyright$. Since $B^n_{vv}\psi_v \leq
  \psi_v$ for all $n$, it follows that
$$
\limsup_n \left(B^n_{vv}\right)^{\frac{1}{n}} \leq 1 .
$$
Hence 
$$
\rho\left( B^{\copyright}\right) = \sup_{v \in \copyright} \limsup_n
\left(B^n_{vv}\right)^{\frac{1}{n}} \leq 1 .
$$
On the other hand, the fact that $B^{\copyright}\psi|_{\copyright} =
\psi|_{\copyright}$ implies that $\rho\left( B^{\copyright}\right)
\geq 1$, and we conclude that
\begin{equation}\label{aa9}
\rho\left( B^{\copyright}\right) = 1.
\end{equation}
Since 
$$
\rho\left( B^{\copyright}\right) = \sup_C \rho\left( B^{C}\right),
$$
where we take the supremum over the components of $G$
contained in $\copyright$, the collection $\mathcal C'$ of
components $C$ from $G$ such that $C \subseteq
\copyright$ and $\rho(B^C) = 1$ is not empty. Order the elements of
$\mathcal C'$ such that $C \leq C'$ when the elements in $C$
talk to the elements of $C'$. Let $\mathcal C$ be the minimal
elements of $\mathcal C'$ with respect to this order. Let
$D\in \mathcal C$. We claim that $D$ is a $B$-harmonic
component, i.e. we assert that
$$
\rho\left( B^{\overline{D} \backslash D}\right) < 1 .
$$ 
Since $\overline{D} \subseteq \copyright$ it follows from (\ref{aa9})
that $\rho\left( B^{\overline{D} \backslash D}\right) \leq 1$. If
$\rho\left( B^{\overline{D} \backslash D}\right) =1$, there must be
one of $G$'s
 components, say $D'$, contained in $\overline{D}
\backslash D$ such that $\rho\left(B^{D'}\right) =1$. But then $D' \in
\mathcal C'$, $D' \neq D$ and $D' \leq D$, contradicting the
minimality of $D$. Hence $D$ is $B$-harmonic as claimed, and we
conclude that $\mathcal C$ consists of $B$-harmonic 
components.

Let $D \in \mathcal C$. Then $B^D\psi|_D \leq \psi|_D$ so it
follows from Perron-Frobenius theory that there is $t_D \geq 0$ such
that $\psi|_D = t_D\psi^D|_D$. Since $\psi|_D$ and $\psi^D|_D$ are
strictly positive, $t_D$ is positive too. Set
$$
\eta = \psi - \sum_{D \in \mathcal C} t_D \psi^D .
$$
We claim that $\eta = 0$. To show this, set $K = \bigcup_{D \in \mathcal C} D$, and note that $\eta|_K
= 0$. Let $H$ be the hereditary closure of $K$, i.e. $H = \widehat{K}$. Consider a $D \in \mathcal C$. When $v \in \left(H
\backslash K\right) \cap \overline{D}$, there is a path from (some
element of) $D' \subseteq K$ to $v$ and a path from $v$ to (some
element of) $D$. Note that $D' \neq D$ since otherwise $v$ would have
to be an element of $D \subseteq K$. But $D' \neq D$ is impossible
since $D$ is minimal for the order on $\mathcal C'$. Hence $\left(H
\backslash K\right) \cap \overline{D} = \emptyset$, showing that $\psi^D|_{H
\backslash K} = 0$.  It follows that $\eta|_{H \backslash K} = \psi_{H
\backslash K}$, and hence that $\eta|_H \geq 0$. Let $w\in H$. There
is an $l \in \mathbb N$ and $v \in K$ such that $B^l_{vw} \neq
0$. Since $B^l\eta = \eta$ we find that $0 = \eta_v = \sum_{u \in V}
B^l_{vu}\eta_u \geq B^l_{vw}\eta_w \geq 0$, implying that $\eta_w =
0$. Hence $\eta|_H = 0$. Now note that 
\begin{equation}\label{a9}
\rho\left(B^{\copyright \backslash H}\right) < 1
\end{equation}
since all components $D$ in $\copyright$ with
$\rho\left(B^D\right) = 1$ are contained in $H$. Since
$$
\left( B^{\copyright \backslash H}\eta\right)_v = \sum_{w \in
  \copyright \backslash H} B_{vw} \eta_w =  \sum_{w \in
  V} B_{vw} \eta_w = \eta_v
$$
for all $v \in \copyright \backslash H$, it follows from (\ref{a9})
that $\eta|_{\copyright \backslash H} = 0$. Thus $\eta = 0$ as claimed
and (\ref{a10}) follows.  

To prove the uniqueness part of the statement let $\mathcal D$ be a
collection of $B$-harmonic components in $G$ and $s_C,C \in
\mathcal D$, positive numbers such that
$$
\psi = \sum_{C \in \mathcal D} s_C \phi^C.
$$
Then $\copyright = \bigcup_{C \in \mathcal C} \overline{C} =
\bigcup_{C \in \mathcal D}\overline{C}$, so when $C \in
\mathcal D$ there is a $C' \in \mathcal C$ such that $C
\cap \overline{C'} \neq \emptyset$. It follows that $C \subseteq
\overline{C'}$ and that either $C' = C$ or $C \subseteq \overline{C'}
\backslash C'$. However, $\rho(B^C) =1$ while $\rho\left(B^{\overline{C'}
  \backslash C'}\right) < 1$, and it follows therefore that $C =
C'$. In this way we conclude that $\mathcal D = \mathcal
C$. Since the preceding argument shows that $C \cap
\overline{C'} = \emptyset$ when $C$ and $C'$ are distinct elements
from $\mathcal C$, we find that
$$
s_C \phi^C|_C = \psi|_C = t_C\phi^C|_C,
$$ 
and hence that $s_C = t_C$ for all $C \in \mathcal C$.
\end{proof}

\begin{cor}\label{a12} The normalized $B$-harmonic vectors
  constitute a finite dimensional simplex whose set of extreme points
  is 
$$
\left\{\phi^C: \ C \ \text{a $B$-harmonic component in $G$}
  \right\} .
$$
\end{cor}

Combining Theorem \ref{a8} with Proposition \ref{b2} we obtain the following

\begin{cor}\label{b1} The set of normalized $B$-almost harmonic
  vectors constitute a finite dimensional simplex whose set of extreme
  points is
$$
\left\{\phi^C: \ C \ \text{a $B$-harmonic component in $G$}
  \right\} \cup \left\{ \phi^s: \ s \ \text{a $B$-summable sink in} \
    G \right\} .
$$
\end{cor}

\section{Gauge invariant KMS states}

It follows from Lemma \ref{a2} and Corollary \ref{b1} that the gauge
invariant $\beta$-KMS states for $\alpha^F$ are determined by the
$A(\beta)$-harmonic components and the $A(\beta)$-summable sinks. In
this section we complete the description of the gauge invariant KMS
states for $\beta \neq 0$ by finding the $A(\beta)$-harmonic components and the
$A(\beta)$-summable sinks for each $\beta \in \mathbb R \backslash
\{0\}$. \footnote{We could have handled the case $\beta = 0$ here also, but it
does simplify things a little when $\beta \neq 0$, and we will have to
consider the $\beta = 0$ case separately for other reasons anyway.}

\subsection{$A(\beta)$-harmonic components}

 A \emph{loop} in $G$ is a finite path $\mu = e_1e_2 \cdots e_n$ (of
positive length, i.e. $n \geq 1$) such that $s(e_1) =
r(e_n)$. If a component $C$ only contains a single loop, we call it \emph{circular}.

%For the following proofs it is convenient to introduce a little more
%notation. When $X$ and $Y$ are sets of vertexes and $n \in
%                 \mathbb{N}$, we let
%$XE^{n}Y$ denote the set of paths $\alpha$ in $G$ of length $n$ such
%that in $s(\alpha) \in X$ and $r(\alpha) \in Y$, 
%and we use
%                 the abbriviation $XEY$ for $XE^{1}Y$ and $XE^{*}Y$
%                 for $\bigcup_{n=1}^{\infty} XE^{n}Y$, and we write
%                 e.g. $XE$ for $XEV$, $E^{n}Y$ for $VE^{n}Y$ and so
%                 on.

\begin{lemma}\label{b4} Let $C \subseteq V$ be a component. The function
$$
\mathbb R \ni \beta \mapsto \rho\left(A(\beta)^C\right)
$$
is log-convex and continuous. 
\end{lemma}
\begin{proof} Since $C$ is a component there is a loop in $C$, of
  length $p$, say. Let $v$ be a vertex on this loop. It follows that $ \log
\rho\left(A(\beta)^C\right) \geq \frac{1}{p} \log
\left(A(\beta)^C\right)^p_{vv}$, showing that the logarithm of the function we consider
takes finite values for all $\beta$. Its continuity
follows therefore from its log-convexity which is established as follows. Let $v \in C$
and $\beta, \beta' \in \mathbb R, t \in [0,1]$. For each $n \in
\mathbb N$ let $vE^nv$ denote the set of paths of length $n$ from $v$
back to itself. Then
$$
\left(A(t\beta  + (1-t)\beta')^C\right)^n_{vv} = \sum_{\mu \in vE^{n}v}
e^{-(t\beta  + (1-t)\beta')F(\mu)} = \sum_{\mu \in vE^{n}v}
\left(e^{-\beta F(\mu)}\right)^t \left(e^{-\beta' F(\mu)}\right)^{1-t} . 
$$
Then
H\"olders inequality shows that
$$
\left(A(t\beta  + (1-t)\beta')^C\right)^n_{vv} \leq
\left(\left(A(\beta)^C\right)^n_{vv}\right)^t
\left(\left(A(\beta')^C\right)^n_{vv}\right)^{1-t} .
$$ 
It follows that 
$$
\rho\left(A(t\beta  + (1-t)\beta')^C\right) = \limsup_n
\left(\left(A(t\beta  +
    (1-t)\beta')^C\right)^n_{vv}\right)^{\frac{1}{n}}
$$
is dominated by the product 
$$
\rho\left(A(\beta)^C\right)^t \rho\left(A(\beta')^C\right)^{1-t} ,
$$
which is what we needed to prove.
\end{proof}

\begin{lemma}\label{b18} Let $C$ be a component in $G$ which is not circular.
\begin{enumerate} 
\item[i)] If $F(\mu) > 0$ for all loops $\mu$ in $C$, there is a
  unique $\beta_0 \in \mathbb R$ such that
  $\rho(A(\beta_0)^C)=1$. This $\beta_0$ is positive and
  $\rho\left(A(\beta)^C\right) < 1$ if and only if $\beta > \beta_0$.
 \item[ii)] If $F(\mu) < 0$ for all loops $\mu$ in $C$, there is a
  unique $\beta_0 \in \mathbb R$ such that
  $\rho(A(\beta_0)^C)=1$. This $\beta_0$ is negative and
  $\rho\left(A(\beta)^C\right) < 1$ if and only if $\beta < \beta_0$. 
\item[iii)] In all other cases, i.e. if $F(\mu) = 0$ for some loop in
  $C$ or there are loops $\mu_1,\mu_2$ in $C$ such that $F(\mu_1) < 0
  < F(\mu_2)$, it follows that $\rho\left(A(\beta)^C\right) > 1$ for
  all $\beta \in \mathbb R$.
\end{enumerate}
\end{lemma}
\begin{proof} Some of the following arguments have appeared in
  \cite{Th3}. i): We claim that $\beta \mapsto
  \rho(A(\beta)^C)$ is strictly decreasing. To see this, set $$
a = \min \left\{ F(\mu) : \ \mu \ \text{is a loop in $C$ of length} \ 
  |\mu| \leq \# C \right\} .
$$ 
Consider $\beta' <\beta$ and a loop $\mu$ in $C$ of length $n$. Then $\mu =
\mu_1\mu_2 \cdots \mu_m$, where each $\mu_i$ is a loop in $C$ of
length $\leq \# C$, and
$$
e^{-\beta' F(\mu)}e^{\beta F(\mu)} = \prod_j e^{(\beta-\beta')
  F(\mu_j)} \geq    e^{m(\beta-\beta')
  a} \geq   e^{\frac{n}{\# C}(\beta-\beta') a} .
$$
Summing over all loops of length $n$ starting and ending at the same
vertex $v$ in $C$, it follows first that
$$
\left(A(\beta')^C\right)^n_{vv} \ \geq \ e^{\frac{n}{\# C}(\beta-\beta')
  a} \left(A(\beta)^C\right)^n_{vv} ,
$$ 
and then that
\begin{equation*}
\begin{split}
&\rho\left( A(\beta')^C\right) =\limsup_n
\left(\left(A(\beta')^C\right)^n_{vv}\right)^{\frac{1}{n}}  \geq \rho\left( A(\beta)^C\right) e^{\frac{1}{\#
      C}(\beta-\beta') a} > \rho\left( A(\beta)^C\right) .
\end{split}
\end{equation*} 
This proves the claim. Note that $A(0)^C$ is the adjacency matrix of
the subgraph $H$ of $G$ whose vertex set is $C$. This is a finite strongly
connected graph and it is
well-known, and easy to show, that $\rho(A(0)^C) > 1$ because $H$ by assumption consists of more
than a single loop. In view of Lemma \ref{b4} it suffices now to show
that $\lim_{\beta \to \infty} \rho(A(\beta)^C) = 0$. To
this end note that any path in $H$ of
  length $\geq \# C$ must visit at least one vertex twice. It follows that for any path $\mu \in P_f(H)$ of length $n$ with $r(\mu) = s(\mu)$ there is a finite collection
$$
\left\{\nu_1,\nu_2,\cdots ,\nu_N\right\} \subseteq \left\{ \mu
  \in P_f(H):  \ 1 \leq |\mu| \leq
  \# C, \ s(\mu) = r(\mu) \right\}
$$
such that $N \geq \frac{n}{\# C}$ and
$$
F(\mu) =\sum_{j=1}^N   F(\nu_j) \geq N a \geq  \frac{n a}{\# C}.
$$  
Let $\beta > 0$ and $v \in C$. Then
$$
A(\beta)^n_{vv} = \sum_{\mu \in vE^{n}v}e^{-\beta F(\mu)}  \leq A(0)^n_{vv} e^{-\frac{\beta a n}{\# C}} ,
$$ 
Hence
$$
\rho\left(A(\beta)^C\right) = \limsup_n
\left(\left(A(\beta)^C\right)^n_{vv}\right)^{\frac{1}{n}} \leq \rho\left(A(0)^C\right)e^{-
    \frac{\beta a}{\# C}} .
$$
Since $ a > 0$, it follows that  $\lim_{\beta \to \infty} \rho(A(\beta)^C) = 0$.
 
The proof of ii) is analogous to that of i).

iii):  Assume first that $F(\mu) = 0$ for some loop in $C$. Since we
assume that $C$ is not circular, there is a path $\nu
$ such that $|\nu| = m|\mu|$ for some $m \in \mathbb N$, $s(\nu) =
r(\nu) = s(\mu)$ and $\nu$ is not the composition of $m$ copies of
$\mu$. It follows that, with $v = s(\mu)$,
$$
\left(A(\beta)^C\right)^{nm|\mu|}_{vv}  \geq
\left(\left(A(\beta)^C\right)^{m|\mu|}_{vv}\right)^n  \geq (e^{-\beta m F(\mu)} + e^{-\beta F(\nu)})^n
=( 1+ e^{-\beta  F(\nu)})^n
$$
for all $n \in \mathbb N$, showing that
$$
\rho\left( A(\beta)^C\right)  \geq  \left(1 + e^{-\beta
    F(\nu)}\right)^{\frac{1}{m|\mu|}} > 1
$$
for all $\beta \in \mathbb R$. Assume then that there are loops $\mu_1,\mu_2$ in $C$ such that $F(\mu_1) < 0
  < F(\mu_2)$. We may assume that $\mu_1$ and $\mu_2$ start at the
  same vertex $v$, if necessary
  after a modification of $\mu_1$ or $\mu_2$. Then
$$
\left(A(\beta)^C\right)^{n|\mu_1||\mu_2|}_{vv} \geq \max \left\{
  e^{-\beta n |\mu_2| F(\mu_1)}, e^{-\beta n |\mu_1| F(\mu_2)} \right\}
$$
for all $n \in \mathbb N$, proving that
$$
\rho\left( A(\beta)^C\right)  \geq  \max \left\{
  e^{-\beta \frac{F(\mu_1)}{|\mu_1|}}, e^{-\beta
    \frac{F(\mu_2)}{|\mu_2|}} \right\} > 1
$$ 
for all $\beta \neq 0$. This completes the proof because $\rho\left( A(0)^C\right) >
1$ since $C$ is not circular.

\end{proof}

\begin{lemma}\label{b19} Let $C$ be a circular component consisting of
  the vertexes in the loop $\mu$. Then
$$
\rho\left(A(\beta)^C\right) = e^{-\beta \frac{F(\mu)}{|\mu|}} 
$$
for all $\beta \in \mathbb R$.
\end{lemma}
\begin{proof} Left to the reader.
\end{proof}

Let $C$ be a component. It follows from Lemma \ref{b18} and Lemma \ref{b19} that when $F(\mu) > 0$ for every loop
$\mu$ in $C$, or $F(\mu) < 0$ for every loop in $C$, there is a unique
number $\beta_C \in \mathbb R$ such that
$$
\rho\left(A(\beta_C)^C\right) = 1.
$$

\begin{defn}\label{b98} A non-circular component $C$ in $G$ is a \emph{KMS component of
  positive type} when 
\begin{enumerate}
\item[i)] $F(\mu) > 0$ for every loop $\mu$ in $\overline{C}$, and 
\item[ii)] $\beta_{C'} < \beta_C$ for every component $C'$ in
  $\overline{C}\backslash C$, if any.
\end{enumerate}
Similarly,  a non-circular component $C$ in $G$ is a \emph{KMS component of
  negative type} when 
\begin{enumerate}
\item[i)] $F(\mu) < 0$ for every loop $\mu$ in $\overline{C}$, and 
\item[ii)] $\beta_C < \beta_{C'}$ for every component $C'$ in
  $\overline{C}\backslash C$, if any. 
\end{enumerate}
\end{defn}

\begin{lemma}\label{april2} 
\begin{enumerate}
\item[i)] Let $\beta > 0$. A non-circular component
  $C$ is $A(\beta)$-harmonic if and only if $C$ is a KMS component of
  positive type and $\beta_C = \beta$.
\item[ii)]  Let $\beta <0$. A non-circular component
  $C$ is $A(\beta)$-harmonic if and only if $C$ is a KMS component of
  negative type and $\beta_C = \beta$.
 \end{enumerate}
\end{lemma}
\begin{proof} The proofs of the two cases are identical and we
  consider here only case i): By definition, $C$ is $A(\beta)$-harmonic if and only if $\rho\left(A(\beta)^C\right) = 1$ and
  $\rho\left(A(\beta)^{\overline{C} \backslash C}\right) < 1$. In view
  of Lemma \ref{b18} the first condition is equivalent to $F(\mu)$
  being strictly positive for every loop $\mu$ in $C$ and
  that $\beta_C = \beta$. Note that 
$\rho\left( A(\beta_C)^{\overline{C} \backslash C}\right) = 0$ when $\overline{C} \backslash C$ is
non-empty, but does not contain any components, while
$$
\rho\left( A(\beta_C)^{\overline{C} \backslash C}\right)  = \max
\left\{ \rho\left(A(\beta_C)^{C'}\right) : \ C' \ \text{a component in
    $\overline{C} \backslash C$} \right\}
$$
otherwise. In view of i) in Lemma \ref{b18} and Lemma \ref{b19} this shows that the second
condition, 
$$
\rho\left(A(\beta_C)^{\overline{C}
      \backslash C}\right) < 1,
$$
holds if and only if $F(\mu) > 0$ for every loop $\mu$ in
$\overline{C} \backslash C$ and $\beta_{C'} <
  \beta_C$ for every component in
    $\overline{C} \backslash C$.
\end{proof}

We consider then the circular components.

\begin{defn}\label{b97} A circular component $C$ in $G$ is a \emph{KMS component of
  positive type} when 
\begin{enumerate}
\item[i)] $F(\nu) = 0$ for the loop $\nu$ in ${C}$,
\item[ii)] $F(\mu)>0$ for all loops $\mu$ in $\overline{C} \setminus C$, if any.
\end{enumerate}
Similarly,  a circular component $C$ in $G$ is a \emph{KMS component of
  negative type} when 
\begin{enumerate}
\item[i)] $F(\nu) = 0$ for the loop $\nu$ in ${C}$, and 
\item[ii)] $F(\mu)<0$ for all loops $\mu$ in $\overline{C} \setminus C$, if any.
\end{enumerate}
\end{defn}

Unlike non-circular components, a circular component $C$ can be a KMS
component of both positive and negative type. This occurs when there
are no loops in $\overline{C} \backslash C$.

Let $C$ be a circular component. Assume that $C$ is a KMS component of positive
type. If there are no components in $\overline{C} \backslash C$, it
follows $\rho\left( A(\beta)^{\overline{C}  \backslash C}\right) = 0$
for all $\beta \in \mathbb R$ and we set $I_C = \mathbb R$ in this case. Otherwise,
set $I_C = \left]\beta_C, \infty\right[$, where 
$$
\beta_C = \max \left\{ \beta_{C'}: \ C' \ \text{a component in} \
  \overline{C} \backslash C \right\}.
$$
Assume then that $C$ is a KMS component of negative type. If there are  no components in $\overline{C} \backslash C$,
we set $I_C = \mathbb R$.  Otherwise,
set $I_C = \left]-\infty, \beta_C\right[$, where 
$$
\beta_C = \min \left\{ \beta_{C'}: \ C' \ \text{a component in} \
  \overline{C} \backslash C \right\}.
$$

In analogy with Lemma \ref{april2} we have the following.

\begin{lemma}\label{april4} 
\begin{enumerate}
\item[i)] Let $\beta > 0$. A circular component
  $C$ is $A(\beta)$-harmonic if and only if $C$ is a KMS component of
  positive type and $\beta \in I_C$.
\item[ii)]  Let $\beta <0$. A circular component
  $C$ is $A(\beta)$-harmonic if and only if $C$ is a KMS component of
  negative type and $\beta \in I_C $.
 \end{enumerate}
\end{lemma}
\begin{proof} Basically the same as for Lemma \ref{april2}.
\end{proof}

\subsection{$A(\beta)$-summable sinks}

\begin{defn}\label{b96} A sink $s$ in $G$ is a \emph{KMS sink of
  positive type} when $F(\mu) >0$ for every loop $\mu$ in $\overline{\{s\}}$, if any, and a \emph{KMS sink of
  negative type} when $F(\mu)< 0$ for every loop $\mu$ in $\overline{\{s\}}$, if any. 
\end{defn}

When there are no loops in $\overline{\{s\}}$ we set $I_s = \mathbb R$. When $s$ is a KMS sink
  of positive type with components in $\overline{\{s\}}$, we set $I_s = \left]\beta_s, \infty \right[$ where
$$
\beta_s = \max \left\{ \beta_{C'} : \ C' \ \text{a component in
    $\overline{\{s\}}$} \right\} .
$$
Similarly, when $s$ is a KMS sink
  of negative type with components in $\overline{\{s\}}$, we set $I_s = \left]-\infty ,\beta_s\right[$ where
$$
\beta_s = \min \left\{ \beta_{C'} : \ C' \ \text{a component in
    $\overline{\{s\}}$} \right\} .
$$

\begin{lemma}\label{april5} 
\begin{enumerate}
\item[i)] Let $\beta > 0$. A sink $s$ in $G$ is $A(\beta)$-summable if
  and only if $s$ is a KMS sink of
  positive type and $\beta \in I_s$.
\item[ii)]  Let $\beta <0$. A sink $s$ in $G$ is $A(\beta)$-summable
  if and only if $s$ is a KMS sink of
  negative type and $\beta \in I_s$.
 \end{enumerate}
\end{lemma}
\begin{proof} Left to the reader.
\end{proof}

\subsection{The gauge invariant  $\beta$-KMS states, $\beta \neq 0$.}

For $\beta \in \mathbb R \backslash \{0\}$, let $\mathcal C(\beta)$ be
the set of non-circular KMS
components $C$ such that $\beta_C =
\beta$, and $\mathcal Z(\beta)$ the set of circular KMS components $D$ such that $\beta \in I_D$. Let $\mathcal S(\beta)$ be the set of KMS
sinks $s$ with $\beta \in I_s$. We can then summarise our findings
with regard to the gauge invariant KMS states as follows.

\begin{thm}\label{gauge} Let $\beta \in \mathbb R \backslash \{0\}$. For every gauge
  invariant $\beta$-KMS state $\varphi$ for $\alpha^F$ there are unique
  functions $f : \mathcal C(\beta) \to [0,1], \ g: \mathcal Z(\beta)
  \to [0,1]$ and $h : \mathcal S(\beta) \to [0,1]$ such that $\sum_C
  f(C) + \sum_D g(D) + \sum_s h(s) = 1$ and
$$
\varphi(S_{\mu}S_{\nu}^*) = \delta_{\mu,\nu} e^{-\beta F(\mu)}
\phi_{r(\mu)}
$$
for all finite paths $\mu, \nu$, where $\phi \in [0,\infty)^V$ is the vector
$$
\phi_v = \sum_{C \in \mathcal C(\beta)} f(C)\phi^C_v + \sum_{D\in
  \mathcal Z(\beta)} g(D)\phi^D_v + \sum_{s \in \mathcal S(\beta)}
h(s)\phi^s_v .
$$
\end{thm}

%For many choices of the function $F$, all KMS states are gauge
%invariant and Theorem \ref{gauge} gives the full story, except in the
%case $\beta = 0$ which we shall deal with below. This is f.ex. the case when $F$ does not annihilate a
%loop, i.e. when $F(\mu) \neq 0$ for every loop $\mu$ in $G$. %It may
%therefore be very tempting to stop the investigations here; in
%particular because the study of the KMS states that are not gauge
%invariant requires a different appraoch. 

\section{Including the KMS states that are not gauge invariant}\label{groupoid}

To handle KMS states that are not gauge invariant we draw on the
results of Neshveyev, \cite{N}. For this it is necessary to introduce the
groupoid picture of $C^*(G)$.

Originally graph $C^*$-algebras were introduced using groupoids,
\cite{KPRR}, but only for row-finite graphs without sinks. For general
graphs the realization as a groupoid $C^*$-algebra was obtained by
A. Paterson in \cite{Pa}. To describe the groupoid for a general graph, possibly infinite
but countable, let $P_f(G)$ and
$P(G)$ denote the set of finite and infinite paths in $G$,
respectively. The range and source maps, $r$ and $s$ on edges, extend in the
natural way to $P_f(G)$; the source map also to $P(G)$. A vertex $v
\in V$ will be considered as a finite path of length $0$ and we set
$r(v) = s(v) =v$ when $v$ is considered as an element of $P_f(G)$. Let
$V_{\infty}$ be the set of vertexes $v$ that are either sinks, or infinite
emitters in the sense that $s^{-1}(v)$ is infinite. The unit space
$\Omega_G$ of $\mathcal G$ is the union
$\Omega_G = P(G) \cup Q(G)$,
where 
$$
Q(G) = \left\{p \in P_f(G): \ r(p) \in V_{\infty} \right\} 
$$ 
is the set of finite paths that terminate at a vertex
in $V_{\infty}$. In particular, $V_{\infty} \subseteq Q(G)$ because
vertexes are considered to be finite paths of length $0$. For any $p
\in P_f(G)$, let $|p|$ denote the length of $p$. When $|p| \geq 1$, set
$$
Z(p) = \left\{ q \in \Omega_G: \ |q| \geq |p| , \ q_i = p_i, \ i = 1,2,
  \cdots, |p| \right\},
$$
and
$$
Z(v) = \left\{ q \in \Omega_G : \ s(q) = v\right\}
$$
when $v \in V$. When $\nu \in P_f(G)$ and $F$ is a finite subset of $P_f(G)$, set
\begin{equation}\label{niels})
Z_F(\nu) = Z(\nu) \backslash \left(\bigcup_{\mu \in F} Z(\mu)\right) .
\end{equation}
The sets $Z_F(\nu)$ form a basis of compact and open subsets for a locally compact Hausdorff
topology on $\Omega_G$. \footnote{Since we here deal with finite
  graphs where there are no infinite emitters, the topology has as an alternative
  basis the sets $Z(\nu)$, corresponding to $Z_F(\nu)$ with $F = \emptyset$.} When $\mu \in P_f(G)$ and $  x \in \Omega_G$, we can define the
concatenation $\mu x \in \Omega_G $ in the obvious way when $r(\mu) =
s(x)$. The groupoid $\mathcal G$ consists of the
elements in $\Omega_G \times \mathbb Z \times \Omega_G$ of the form
$$
(\mu x, |\mu| - |\mu'|, \mu'x),
$$
for some $x\in \Omega_G$ and some $ \mu,\mu' \in P_f(G)$. The product
in $\mathcal G$ is defined by
$$
(\mu x, |\mu| - |\mu'|, \mu' x)(\nu y, |\nu| -|\nu'|, \nu' y) = (\mu
x, \ |\mu | + |\nu| - |\mu'| - |\nu'|, \nu' y),
$$ 
when $\mu' x = \nu y$, and the involution by $(\mu x, |\mu| - |\mu'|,
\mu'x)^{-1} = (\mu' x, |\mu'| - |\mu|, \mu x)$. To describe the
topology on $\mathcal G$, let $Z_{F}(\mu)$ and $Z_{F'}(\mu')$ be two
sets of the form (\ref{niels}) with $r(\mu) = r(\mu')$. The topology we
shall consider has as a basis the sets of the form
\begin{equation}\label{top}
\left\{ (\mu x, |\mu| - |\mu'|, \mu' x) : \ \mu x \in Z_F(\mu), \
  \mu'x \in Z_{F'}(\mu') \right\} .
\end{equation}
With this topology $\mathcal G$ becomes an \'etale second countable locally compact Hausdorff groupoid and we can consider the reduced $C^*$-algebra $C^*_r(\mathcal
G)$ as in
\cite{Re}. As shown by Paterson in
\cite{Pa} there is an isomorphism $C^*(G) \to
  C^*_r(\mathcal G)$ which sends $S_e$ to $1_e$, where $1_e$ is the
  characteristic function of the compact and open set
$$
\left\{ (ex, 1, r(e)x) : \ x \in \Omega_G \right\} \ \subseteq \
\mathcal G,
$$  
and $P_v$ to $1_v$, where $1_v$ is the characteristic function of the
compact and open set
$$
\left\{ (vx,0,vx) \ : \ x \in \Omega_G \right\} \ \subseteq \ \mathcal G.
$$ 
In the following we use the identification $C^*(G) = C_r^*(\mathcal
G)$ and identify $\Omega_G$ with the unit space of $\mathcal G$ via
the embedding
$\Omega_G \ni x \ \mapsto \ ( x, 0,x)$. In this way we get a canonical
embedding $C(\Omega_G) \subseteq C^*(G)$ and there is a conditional
expectation $P : C^*(G) \to C(\Omega_G)$ defined such that
$$
P(f)(x) = f(x,0,x)
$$
when $f \in C_c(\mathcal G)$, cf. \cite{Re}. This conditional
expectation can be used to characterise the gauge invariant KMS states
because it follows from Theorem 2.2 in \cite{Th2} that a KMS state for
$\alpha^F$ is
gauge invariant if and only if it factorises through $P$.

%\begin{remark}\label{a25}
%We note that the operator $S_{\mu}S_{\nu}^*$ coming from two finite
%paths $\mu,\nu \in P_f(G)$ with $r(\mu) = s(\nu)$, in the groupoid
%picture is given by the characteristic function of the set
%$$
%\left\{ (\mu x, |\mu| - |\nu|, \nu x) : \ x \in \Omega_G, \ s(x) =
%  r(\mu) = r(\nu) \right\} .
%$$
%\end{remark}

To describe the automorphism group $\alpha^F$ in the groupoid picture
we define a continuous homomorphism $c_{F} : \mathcal G \to \mathbb R$ by
$$
c_{F}(ux,|u|-|u'|,u'x) = F(u)- F(u').
$$
The automorphism group $\alpha^F$ on $C^*_r(\mathcal G)$ is then 
defined such that
$$
\alpha^{F}_t(f)(\gamma) = e^{it c_{F}(\gamma)} f(\gamma)
$$
when $f \in C_c(\mathcal G)$, cf. \cite{Re}.

Thanks to this picture of $C^*(G)$ and $\alpha^F$, and because we
consider finite graphs in this paper, we can draw
on the results of Neshveyev, \cite{N}, to obtain a decomposition of
the KMS states into those that are gauge invariant and those that are
not. Since the groupoid $\mathcal G$ has
the additional properties required in Section 2 of \cite{Th1} we can
use the description obtained in Theorem 2.4 of \cite{Th1} when $\beta
\neq 0$. Of the $\beta$-KMS states considered in Theorem 2.4 in
\cite{Th1}, it is only those of the form
$\omega^{\varphi}_{\mathcal O}$
which may not factor through $P$. Here $\mathcal O$ is an orbit in
$\Omega_G$ under the canonical action of the groupoid $\mathcal G$ on its unit
space, and $\mathcal O$ must be consistent and $\beta$-summable for
$\omega^{\varphi}_{\mathcal O}$ to be defined. Furthermore, the
formula for $\omega^{\varphi}_{\mathcal O}$ shows that it is only
if the points in $\mathcal O$ have non-trivial isotropy group in
$\mathcal G$ that $\omega^{\varphi}_{\mathcal O}$ does not factor
through $P$.

Note that the isotropy group $\mathcal G^x_x \subseteq \mathcal G$ of
an element $x\in \Omega_G$ is trivial unless $x$ is an infinite
path in $G$ which is pre-period. Its orbit under $\mathcal G$ is
then the orbit of an infinite periodic path. We may
therefore assume that there is a loop $\delta$ in $G$ such that $x = \delta^{\infty} \in P(G)$. Then
$$
\mathcal G_x^x = \left\{ (x,kp,x) : \ k \in \mathbb Z \right\} ,
$$
where $p$ is the period of $\delta^{\infty}$. We may assume that $p =
|\delta|$ and find then that $c_F(x,kp,x) = {kF(\delta)}$. It follows 
that the $\mathcal G$-orbit $\mathcal
Gx$ is consistent in the sense used in \cite{Th1} if and only if
$F(\delta) = 0$. If the component of $G$ containing $\delta$ contains
a second loop, there will be another loop $\delta'$ in $G$ starting and
ending at the same vertex as $\delta$. Then
$$
x_n = \delta^n \delta' \delta^{\infty} , \ n \in \mathbb N,
$$   
are distinct elements in $\mathcal Gx$, and when we use the notation from \cite{Th1}, we have that
$$
l_x(x_n) = e^{-F(\delta')} .
$$
This shows that 
$$
\sum_{z \in \mathcal Gx} l_x(z)^{\beta} =
\infty
$$ 
for all $\beta \in \mathbb R$, and we conclude therefore that
$\mathcal G x$ is not $\beta$-summable for any $\beta \in \mathbb R$. It follows that the only $\mathcal G$-orbits of elements with
non-trivial isotropy groups which can be both consistent and
$\beta$-summable in the sense of \cite{Th1}, are the $\mathcal
G$-orbits of a periodic infinite path lying in a circular component
consisting of a loop $\delta$ with $F(\delta) = 0$. On the other hand, for such an infinite path $x$ the
corresponding $\mathcal
G$-orbit will be $\beta$-summable if and only if 
\begin{equation}\label{b36}
\sum_{\mu \in E^*_{\delta}s(x)} e^{-\beta F(\mu)} < \infty  , 
\end{equation}
where $E^*_{\delta}s(x)$ denotes the set of finite paths $\mu$ in $G$ that terminate at
$s(x) \in V$ and do not contain $\delta$. Note that
(\ref{b36}) will hold if and only if $C$ is a circular KMS component
with $\beta \in I_C$. In this case the
$\beta$-KMS state $\omega^{\varphi}_{\mathcal O}$ is defined for every
state $\varphi$ on $C^*\left(\mathcal G^x_x\right)$, but it will only
be extremal when $\varphi$ is a pure state. By using the
identification $C^*(\mathcal G_x^x) = C(\mathbb T)$ this means that the
extremal $\beta$-KMS states occurring in Theorem 2.4 in \cite{Th1} that are not
gauge invariant arise from a number $\lambda \in \mathbb T$,
considered as a pure state on $C(\mathbb T)$, and a component $C$ of
zero type with $\beta \in I_C$. We will denote this extremal $\beta$-KMS state
by $\omega^{\lambda}_C$. The
formula for this state, as it was given in \cite{Th1}, becomes 
\begin{equation}\label{b30}
\begin{split}
& \omega^{\lambda}_{C} (f) = 
  \left( \sum_{\nu \in E^*_{\delta}s(x)} e^{-\beta F(\nu)} \right)^{-1} \sum_{k
    \in \mathbb Z} \sum_{\mu \in E^*_{\delta}s(x)} \lambda^k e^{-\beta F(\mu)}
  f\left(\mu x,kp,\mu x\right) \\
%& = \left( \sum_{\nu \in E^*v} e^{-\beta F(\nu)} \right)^{-1} \left(\sum_{\mu \in E^*x} e^{-\beta F(\mu)} f\left(\mu
%  \delta^{\infty},0,\mu \delta^{\infty}\right) + \sum_{k \in \mathbb Z
%  \backslash \{0\}} \lambda^k  f\left(\delta^{\infty},kp,\delta^{\infty}\right)\right)
\end{split} 
\end{equation}
when $f \in C_c(\mathcal G)$. A general state $\varphi$ on $C(\mathbb T)$ is
given by integration against a Borel probability measure $\mu$ on $\mathbb T$ and the
corresponding $\beta$-KMS state $\omega^{\varphi}_{\mathcal O}$ from
\cite{Th1}, which we in the present setting will denote by
$\omega^{\mu}_C$, is then given as an
integral
\begin{equation}\label{b83}
\omega^{\mu}_C(a) = \int_{\mathbb T} \omega^{\lambda}_C(a) \
d\mu(\lambda) .
\end{equation}
The conclusions we need here can then be summarised in the following way.

\begin{lemma}\label{b49} Let $\beta \in \mathbb R \backslash
  \{0\}$. For every $\beta$-KMS state $\varphi$ for $\alpha^F$ there
  is a Borel probability measure $\nu$ on $\Omega_G$, Borel
  probability measures $\mu_{D}, D \in \mathcal Z(\beta)$, on
  $\mathbb T$ and numbers $t$ and $t_{D}, D \in \mathcal Z(\beta)$, in
  $[0,1]$ such that $t + \sum_{D \in \mathcal Z(\beta)} t_D = 1$ and
\begin{equation}\label{b48}
\varphi(a) = t \int_{\Omega_G} P(a) \ d\nu + \sum_{D \in \mathcal
  Z(\beta)} t_{D} \omega^{\mu_{D}}_{D}(a) .
\end{equation}
The numbers $t$ and $t_D$ are uniquely determined by $\varphi$, as are
the Borel probability measures $\mu_D$ with $t_D > 0$.
\end{lemma}

The measure $\nu$ in Lemma \ref{b49} have certain properties which reflect that $\varphi$
is a KMS state, and they can be found in \cite{Th1}, but what matters
here is only that 
$$
a \mapsto  \int_{\Omega_G} P(a) \ d\nu
$$
is $\beta$-KMS state which is gauge invariant. It is therefore a
convex combination of the states $\varphi_C, \varphi_s, \varphi_D$
given by the formula (\ref{a3}) when the vector $\psi$ occurring there
is substituted by the $A(\beta)$-almost harmonic vectors $\phi^C, C \in \mathcal
C(\beta), \phi^s, s \in \mathcal S(\beta),$ and $\phi^D, D \in
\mathcal Z(\beta)$, respectively. Note that the state $\varphi_D$
corresponding to a component $D \in \mathcal Z(\beta)$ is the same
as the state $\omega^m_D$ from (\ref{b83}) when $m$ is the normalized Lebesgue measure
on $\mathbb T$. We can therefore now use
Theorem 2.4 in \cite{Th1} and combine Lemma \ref{b49} with Theorem
\ref{gauge} to obtain the following description of the $\beta$-KMS
states when $\beta \neq 0$. 

\begin{thm}\label{MAIN} For $\beta  \in \mathbb R \backslash \{0\}$, 
\begin{enumerate}
\item[$\bullet$] let $\mathcal C(\beta)$ be
  the set of non-circular KMS components $C$ in $G$ with $\beta_C = \beta$,
\item[$\bullet$] let $\mathcal S(\beta)$ be the set of KMS sinks $s$ in $G$
  with $\beta \in I_s$, and
\item[$\bullet$] let $\mathcal Z(\beta)$ be the set of circular KMS components $D$ with
  $\beta \in I_D$. 
\end{enumerate}
For every $\beta$-KMS state $\varphi$ for
  $\alpha^F$ there are numbers
  $\alpha_C \in [0,1], C \in \mathcal C(\beta)$, $\alpha_s \in [0,1],
  s \in \mathcal S(\beta)$, and $\alpha_{D} \in [0,1],D \in \mathcal
  Z(\beta)$, and Borel probability measures $\mu_{D}, D \in \mathcal
  Z(\beta)$, on $\mathbb T$, such that $\sum_C\alpha_C
   + \sum_s \alpha_s + \sum_D\alpha_D = 1$, and
$$
\varphi \ = \sum_{C \in \mathcal C(\beta)} \alpha_C\varphi_C \ + \  \sum_{s \in
  \mathcal S(\beta)} \alpha_s\varphi_s  \ + \ \sum_{D \in \mathcal Z(\beta)}
\alpha_{D} \omega^{\mu_{D}}_{D} .
$$
The numbers $\alpha_C, \alpha_s, \alpha_D$ are uniquely
determined by $\varphi$, as are the Borel probability measures $\mu_{D}$
for the components $D \in \mathcal Z(\beta)$ with $\alpha_{D} > 0$.
\end{thm}

%The following section is devoted to the $\beta =0$ case of Theorem \ref{MAIN}.

\subsection{Trace states}\label{traces}

We need a different approach when $\beta =0$. Since the $0$-KMS
states are the trace states of $C^*(G)$ we must determine these. 

Let $\mathcal Z(0)$ denote the set of circular components $C$ in $G$
with the property that $\overline{C}\backslash C$ does not contain any
components, and similarly $\mathcal S(0)$ the set of sinks $s$ in $G$
such that $\overline{\{s\}} \backslash \{s\}$ does not contain a
component. For every $C \in \mathcal Z(0)$ the set $V \backslash
\overline{C}$ is hereditary and saturated, and there is a surjective
$*$-homomorphism $\pi_C : C^*(G) \to C^*(\overline{C})$, where
$\overline{C}$ is considered as a directed graph with vertex set
$\overline{C} \subseteq V$ and the edge set $\left\{e \in E: \
  s(e),r(e) \in \overline{C} \right\}$, cf. Theorem 4.1 in
\cite{BPRS}. Similarly, when $s\in \mathcal S(0)$ there is also a
surjective $*$-homorphism  $\pi_s : C^*(G) \to C^*(\overline{\{s\}})$, where
$\overline{\{s\}}$ is considered as a directed graph with vertex set
$\overline{\{s\}} \subseteq V$ and the edge set $\left\{e \in E: \
  s(e),r(e) \in \overline{\{s\}} \right\}$. 

When $s \in \mathcal S(0)$ we let $n_s$ be the number of paths
in $G$ terminating at $s$. When $C \in \mathcal Z(0)$ we choose a
vertex $v_C \in C$ and set
$$
n_C = \# \left\{ \mu \in P_f(G): \ r(\mu) = v_C, \ s(\mu_i) \neq
  v_C, \ \text{for} \ i \leq |\mu| \right\} ,
$$ 
where the condition that $ s(\mu_i) \neq
  v_C$ is negligible when $|\mu| = 0$.

%The following theorem completes the proof of Theorem \ref{MAIN}.

\begin{thm}\label{L11} For every $s \in \mathcal S(0)$,
$$
C^*(\overline{\{s\}}) \simeq M_{n_s}(\mathbb C), 
$$
and for every $C \in \mathcal Z(0)$,
$$
C^*(\overline{C}) \simeq M_{n_C}\left(C(\mathbb T)\right), \ C \in
\mathcal Z(0).
$$
For every trace state $\omega$ on $C^*(G)$ there are unique 
numbers $\alpha_s \in [0,1]$ and $\alpha_C \in [0,1]$, and trace states $\omega_s$ on $C^*(\overline{\{s\}})$
and $\omega_C$ on $C^*(\overline{C})$, $s \in \mathcal S(0), \ C \in
\mathcal Z(0)$, such that
$$
\sum_{s \in \mathcal S(0)} \alpha_s + \sum_{C \in \mathcal Z(0)}
\alpha_C = 1
$$
and
$$
\omega = \sum_{s \in \mathcal S(0)} \alpha_s \omega_s \circ \pi_s +
\sum_{C \in \mathcal Z(0)} \alpha_C \omega_C \circ \pi_C .
$$
\end{thm}

For the proof of Theorem \ref{L11} set 
$$
N = V \backslash \left(\bigcup_{C \in \mathcal Z(0)} \overline{C} \cup \bigcup_{s
  \in \mathcal S(0)} \overline{\{s\}}\right) .
$$
Then $N$ is hereditary and saturated, and the set $\left\{P_v : \ v \in
  N\right\}$ generates an ideal $I_N$ in $C^*(G)$ such that
$C^*(G)/I_N \simeq C^*(\tilde{G})$ where $\tilde{G}$ is the graph with
vertex set 
$$
\tilde{V} =  \bigcup_{C \in \mathcal Z(0)} \overline{C} \cup \bigcup_{s
  \in \mathcal S(0)} \overline{\{s\}}
$$ 
and edge set $\tilde{E} = \left\{e\in E : \ r(e)
\notin N \right\}$, cf. Theorem 4.1 in \cite{BPRS}.

\begin{lemma} \label{L12}  
Let $\omega$ be a trace state on
  $C^*(G)$. Then $\omega(I_N) = 0$.
\end{lemma}
\begin{proof}  It suffices to show that $\omega(P_v) = 0$ when $v \in
  N$. To this end consider a loop $\mu$ in $G$ with vertexes $v_{1}, \dots, v_{n},v_1$. The Cuntz-Krieger relations (\ref{CKrel}) imply:
\begin{align*}
\omega(P_{v_1}) &= \omega \big( \sum_{e \in s^{-1}(v_{1})} S_{e}S_{e}^{*} \big) =  \sum_{e \in s^{-1}(v_{1})} \omega(S_{e}^{*}S_{e}) = \sum_{e \in s^{-1}(v_{1})} \omega(P_{r(e)}) \geq \omega(P_{v_{2}})\\
& = \sum_{e \in s^{-1}(v_2)}\omega(P_{r(e)}) \geq \omega(P_{v_3}) =
\cdots \geq \omega(P_{v_n}) =  \sum_{e \in s^{-1}(v_n)} \omega(P_{r(e)}) \geq \omega(P_{v_1}) .
\end{align*}
Hence we must have equality everywhere, which implies that
$\omega(P_{r(e)})=0$ if $e \in s^{-1}(v_{i})$ for some $i$, but
$e \notin \mu$. It follows from this that $\omega(P_w) = 0$ when
$$
w \in \bigcup_{C \in \mathcal C} \widehat{C} \ \backslash \bigcup_{C'
  \in \mathcal Z(0)} C' 
$$ 
where $\mathcal C$ is the set of components. Hence if $s$ is a sink in $G$ it follows that $\omega(P_s) = 0$ unless
$s \in \mathcal S(0)$. Consider a vertex $v \in N$. If $v$ is sink, $\omega(P_v) = 0$
and we are done. Otherwise, if $\omega(P_v) > 0$, the Cuntz-Krieger relations (\ref{CKrel})
implies that there is an edge $e_1\in s^{-1}(v)$ such that
$\omega(P_{r(e_1)}) > 0$. Then $r(e_1)$ can not be a sink and we can
find an edge $e_2$ such that $s(e_2) = r(e_1)$ and $\omega(P_{r(e_2)})
> 0$. We can continue this construction of edges $e_i$ indefinitely so there are $i < i'$ such that $s(e_i) = r(e_{i'})$, and
the path $e_ie_{i+1} \cdots e_{i'}$ is contained in a
component $C$. Since $\omega(P_{r(e_i)}) > 0$ this component must be
circular and without components in $\overline{C} \backslash C$, which contradicts that $v \in N$. It follows that
$\omega(P_v) = 0$.  
\end{proof}

 For each $C \in \mathcal{Z}(0)$, fix a vertex $v_{C} \in C$, and set $v_{s} =s$ for $s \in \mathcal{S}(0)$. For all $v \in
 \tilde{V}$ and $a \in \mathcal Z(0) \cup \mathcal S(0)$, we define:
\begin{equation*}
N^{a}_{v} = \{ \mu \in P_f(\tilde{G}) \ | \ s(\mu)=v \ , \ r(\mu)=v_{a}
\ , \ s(\mu_{i})\neq v_{a} \text{ for } i \leq \lvert \mu \rvert\} 
\end{equation*}
where the condition that $s(\mu_{i})\neq v_{a}$ is negligible when $\lvert \mu \rvert = 0$. We define $N^{a} = \bigcup_{v \in \tilde{V}} N^{a}_{v}$ for $a \in  \mathcal Z(0) \cup \mathcal S(0)$. 
\begin{lemma}\label{L13}
\begin{equation*}
C^{*}(\tilde{G}) \simeq  \Big(\bigoplus_{s \in \mathcal{S}(0)}M_{\#N^{s}}(\mathbb{C}) \Big) \oplus \Big( \bigoplus_{C \in \mathcal{Z}(0)} M_{\#N^{C}}(C(\T))\Big)
\end{equation*}
\end{lemma}
\begin{proof} For $a \in  \mathcal Z(0) \cup \mathcal S(0)$, let $e_{\alpha, \beta}, \alpha,\beta \in
  N^a$ be the standard matrix units in $M_{N^a}(\mathbb C) \simeq
  M_{\# N^a}(\mathbb C)$. For $v \in \tilde{V}$,
set
$$
\tilde{P}_v = \sum_{a \in \mathcal Z(0) \cup \mathcal S(0)} \sum_{\alpha \in N^a_v} e_{\alpha,\alpha} .
$$
Then $\tilde{P}_v, v \in \tilde{V}$, are mutually orthogonal projections. For each
$f \in \tilde{E}$ such that $s(f) \notin \left\{ v_a: \ a \in
  \mathcal Z(0) \cup \mathcal S(0)\right\}$, set
$$
\tilde{S}_f = \sum_{a \in \mathcal Z(0) \cup \mathcal S(0)} \sum_{\alpha \in N^a_{r(f)}} e_{f\alpha
  ,\alpha} .
$$
If $s(f) \in \left\{ v_a: \ a \in
   \mathcal Z(0) \cup \mathcal S(0)\right\}$, then $s(f) = v_C$ for some $C \in \mathcal Z(0)$, and we
let $\mu^C$ denote the unique shortest path in $\tilde{G}$ with $s(\mu^C) =
r(f)$ and $r\left(\mu^C\right) = v_C$. We define an element
$$
\tilde{S}_f \in C\left(\mathbb T, M_{N^C}(\mathbb C)\right)
$$
such that
$$
\tilde{S}_f(z) = ze_{v_C, \mu_C} .
$$
It is straightforward to verify that $\tilde{P}_v, v \in \tilde{V}$,
and $\tilde{S}_f, f \in \tilde{E}$, is a Cuntz-Krieger family,
i.e. they satisfy (\ref{CKrel}) relative to $\tilde{G}$. Since
$$
\tilde{P}_v, \tilde{S}_f \in  \Big(\bigoplus_{s \in
  \mathcal{S}(0)}M_{\#N^{s}}(\mathbb{C}) \Big) \oplus \Big(
\bigoplus_{C \in \mathcal{Z}(0)} M_{\#N^{C}}(C(\T))\Big)
$$
for all $v \in \tilde{V}$ and all $f\in \tilde{E}$, the universal
property of $C^*(\tilde{G})$ gives us a canonical $*$-homomorphism  
$$
C^*(\tilde{G})  \ \to \ \Big(\bigoplus_{s \in
  \mathcal{S}(0)}M_{\#N^{s}}(\mathbb{C}) \Big) \oplus \Big(
\bigoplus_{C \in \mathcal{Z}(0)} M_{\#N^{C}}(C(\T))\Big) .
$$
To show that this is an isomorphism, note first that it is surjective
because the target algebra is generated as a $C^*$-algebra by $\tilde{P}_v, v \in \tilde{V}$,
and $\tilde{S}_f, f \in \tilde{E}$. For the injectivity we shall
appeal to the gauge-invariant uniqueness theorem, Theorem 2.1 in
\cite{BPRS}. For a $a \in \mathcal{S}(0) \cup \mathcal{Z}(0)$, define for each $\omega \in \T$ the unitary:
$$
U^{a}_{\omega} = \sum_{\alpha \in N^{a}} \omega^{\lvert \alpha \rvert} e_{\alpha , \alpha}
$$
For $s \in \mathcal{S}(0)$ we define an automorphism
$\psi^{s}_{\omega}$ on $M_{\# N^{s}}(\mathbb{C})$ by
$\psi^{s}_{\omega}(A)=U_{\omega}^{s} A U_{\overline{\omega}}^{s}$, and
for $C \in \mathcal{Z}(0)$ we define an automorphism on $M_{\# N^{C}}(C(\T))$ by $\psi_{\omega}^{C}(f)(z)=U_{\omega}^{C}f(\omega^{\#C} z) U_{\overline{\omega}}^{C}$. It is straightforward to check that:
$$
\T \ni \omega \to \psi_{\omega}:=(\bigoplus_{s \in \mathcal{S}}\psi_{\omega}^{s}) \oplus (\bigoplus_{C \in \mathcal{Z}(0)}\psi_{\omega}^{C})
$$
is an action, and that we for $f\in \tilde{E}$ and $v\in \tilde{V}$ have:
$$
\psi_{\omega} (\tilde{S}_f) = \omega \tilde{S}_f \qquad \qquad \psi_{\omega}(\tilde{P}_v) = \tilde{P}_v
$$
for all $\omega \in \T$. It follows therefore from Theorem 2.1 in \cite{BPRS}
that the homomorphism under consideration is injective.

% For $z \in \mathbb T$, set
%$$
%U_z = \sum_{a \in \mathcal Z(0) \cup \mathcal S(0)} \sum_{\alpha \in N^{a}} z^{\lvert \alpha \rvert} e_{\alpha , \alpha},
%$$
%which is a unitary in the target algebra. It is straightforward to
%check that
%$$
%U_z\tilde{S}_fU_{-z} = z\tilde{S}_f
%$$
%for all $f\in \tilde{E}$ and
%$$
%U_z\tilde{P}_vU_{-z} = \tilde{P}_v
%$$
%for all $v\in \tilde{V}$. It follows therefore from Theorem 2.1 in \cite{BPRS}
%that the homomorphism under consideration is injective.

\end{proof}

\emph{Proof of Theorem \ref{L11}:} Consider $C \in \mathcal Z(0)$ and
let $C^*(G) \to M_{\# N^C}(C(\mathbb T))$ be the surjective
$*$-homomorphism obtained by composing the quotient map $C^*(G) \to
C^*(\tilde{G})$ with the projection $C^*(\tilde{G}) \to  M_{\#
  N^C}(C(\mathbb T))$ obtained from Lemma \ref{L13}. %$\overline{C}$ is
%a graph in itself, and Lemma \ref{L12} and Lemma \ref{L13} hold for
%$\overline{G}$ in place of $G$. For $\overline{G}$ the $N$ from Lemma \ref{L12} is empty, and since $N^{C}$ doesn't depend on whether we consider vertexes in $G$ or in $\overline{G}$, Lemma \ref{L13} gives that $C^*( \overline{C}) \simeq  M_{\# N^C}(C(\mathbb
%T))$. In the same way we see that $C^*(\overline{\{s\}}) \simeq M_{\#  N^s}(\mathbb C)$, and the statements
%regarding a trace state $\omega$ follow from Lemma \ref{L12} and Lemma \ref{L13}.
The kernel of
this $*$-homomorphism is the same as the kernel of $\pi_C : C^*(G) \to
C^*(\overline{C})$, namely the ideal generated by
$$
\left\{ P_v : \ v \notin \overline{C} \right\} .
$$
It follows that $C^*(\overline{C}) \simeq  M_{\# N^C}(C(\mathbb
T))$. In the same way we see that $C^*(\overline{\{s\}}) \simeq M_{\#
  N^s}(\mathbb C)$ when $s \in \mathcal S(0)$. The statements
regarding a trace state $\omega$ follow from Lemma \ref{L12} and Lemma
\ref{L13}.
\qed

\section{Ground states}

To describe the ground states we use again the groupoid picture
described in Section \ref{groupoid} in order to adapt the approach
from Section 5 in \cite{Th4} to the present setting. The fixed point algebra of $\alpha^F$ is the $C^*$-algebra of the open sub-groupoid
$$
\mathcal F = \left\{ (\mu x, |\mu| - |\mu'|, \mu'x)  : \ x \in
  \Omega_G, \ F(\mu) = F(\mu')  \right\} 
$$  
of $\mathcal G$.
The conditional expectation 
$$
Q :C^*(G) \to C^*_r(\mathcal F)
$$ 
extending the
restriction map $C_c(\mathcal G) \to C_c(\mathcal F)$ can be
described as a limit:
\begin{equation}\label{intX}
Q(a) = \lim_{R \to \infty} \frac{1}{R} \int_{0}^{R} \alpha^F_t(a) \ dt ,
\end{equation}
cf. the proof of Theorem 2.2 in \cite{Th3}.

When $x \in \Omega_G, z \in P_f(G)$, write $z \subseteq x$ when $1 \leq |z|$ and $x|_{[1,|z|]} = z$ or $|z| =0$ and $z = s(x)$. An
element $x \in \Omega_G$ \emph{has minimal $F$-weight} when the following holds:
$$
z,z' \in P_f(G), \ z \subseteq x, \ r(z') = r(z) \ \Rightarrow \ F\left(z'\right) \geq  F\left(z\right) .
$$ 
We denote the set of elements in $\Omega_G$ with minimal $F$-weight by
$\text{Min}(F,G)$. Then $\text{Min}(F,G)$ is closed in $\Omega_G$ and
$\mathcal F$-invariant in the sense that 
$$
(x,k,y)  \in \mathcal F, \ x \in \text{Min}(F,G)  \ \Rightarrow \ y \in \text{Min}(F,G) .
$$
It follows that the reduction $\mathcal F|_{\text{Min}(F,G)}$ of
$\mathcal F$ to $\text{Min}(F,G)$, defined by
$$
\mathcal F|_{\text{Min}(F,G)} = \left\{ (\mu x, |\mu| - |\mu'|,\mu' x)
  : \ x \in \Omega_G, \ F(\mu) = F(\mu'), \ \mu x \in \text{Min}(F,G)
\right\},
$$
is
a locally compact \'etale groupoid. Furthermore, there is a surjective
$*$-homomorphism 
$$
R : C^*_r(\mathcal F) \to C_r^*\left( \mathcal
  F|_{\text{Min}(F,G)}\right)
$$ 
extending the restriction map $C_c(\mathcal F) \to C_c\left( \mathcal
  F|_{\text{Min}(F,G)}\right)$. Now the proof of Theorem 5.3 in
\cite{Th4} can be repeated almost ad verbatim to yield the following.

\begin{thm}\label{L20} The map $\omega \mapsto \omega \circ R \circ Q$
  is an affine homeomorphism from the state space of $ C^*_r\left( \mathcal
  F|_{\text{Min}(F,G)}\right)$ onto the ground states of $\alpha^F$.
\end{thm} 

The structure of the $C^*$-algebra $ C_r^*\left( \mathcal
  F|_{\text{Min}(F,G)}\right)$ varies a lot with the choice of
$F$. When $F$ is constant zero,
it is equal to $C^*(G)$, and when $F$ is strictly positive it is
isomorphic to $\mathbb C^n$, where $n$ is the number of sinks in
$G$. If $G$ consists of three edges, $e_i$, and a vertex $v$ with $r(e_i) =
s(e_i) = v, i =1,2,3$, and if $F(e_1) = F(e_2) = 0$ while $F(e_3) =1$,
we find that $C^*(G)$ is the Cuntz-algebra $O_3$ while $ C_r^*\left( \mathcal
  F|_{\text{Min}(F,G)}\right)$ is a copy of $O_2$. 

Which of the ground states are weak* limits, for $\beta \to \infty$,
of $\beta$-KMS states, can be decided by combining Theorem \ref{L20}
with Theorem \ref{MAIN}. It follows, for instance, that they all are
when $F = 1$, while none of them are in the last mentioned example.

\section{An example}

Consider the following graph $G$. The two sinks are $s_1$ and $s_2$
and there are four components labelled $C_1$ through $C_4$. In order to
define various functions on the edge set we have labelled four edges
$a,b,c$ and $d$.

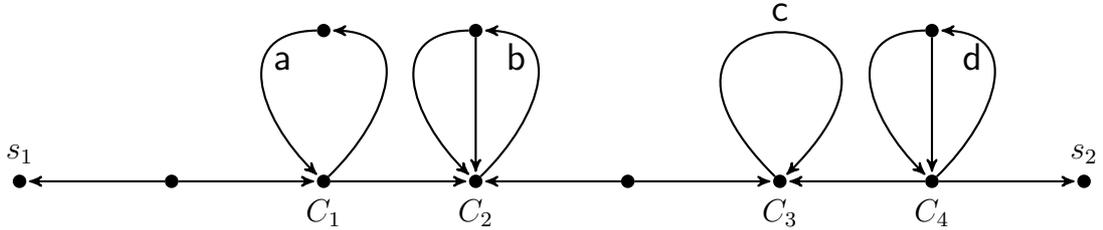
\tikzstyle{input} = [draw, circle, fill, inner sep=1.5pt]
\begin{tikzpicture}[->,>=stealth',shorten >=1pt,auto,node distance=2cm,
                    thick,main node/.style={circle,draw,font=\sffamily\small\bfseries}]

  \node[input](1) [label=above:$s_{1}$]{} ;
  \node[input] (2) [right of =1] {};
  \node[input] (3) [right of =2] [label=below:$C_{1}$]{};
  \node[input] (4) [above of =3] {};
  \node[input] (5) [right of =3] [label=below:$C_{2}$] {};
  \node[input] (6) [above of =5] {};
  \node[input] (7) [right of =5]  {};
  \node[input] (8) [right of =7] [label=below:$C_{3}$] {};
  \node[input] (9) [right of =8] [label=below:$C_{4}$] {};
  \node[input] (10) [above of =9] {};
  \node[input] (11) [right of =9] [label=above:$s_{2}$] {};

  \path[every node/.style={font=\sffamily\large}, every loop/.style={min distance=36mm,in=45,out=135,looseness=10}]
    (2)  edge  node {} (1)
    (2)  edge node {} (3)
    (3)  edge [bend right,  looseness=1.5, out=-45, in=-90]node {} (4)
    (4)  edge [bend right, looseness=1.5, out=-90, in=-135] node {a} (3)
    (3)  edge  node {} (5)
    (5)  edge [bend right, looseness=1.5, out=-45, in=-90] node {b} (6)
    (6)  edge node {} (5)
    (6)  edge [bend right, looseness=1.5, out=-90, in=-135] node {} (5)
    (7)  edge node {} (5)
    (7)  edge node {} (8)
    (9)  edge node {} (8)
    (8)  edge  [loop] node {c} (8)
    (9)  edge node {} (11)
    (9)  edge [bend right, looseness=1.5, out=-45, in=-90] node {d} (10)
    (10)  edge node {} (9)
    (10)  edge [bend right, looseness=1.5, out=-90, in=-135] node {} (9);
\end{tikzpicture}

\smallskip

Consider first the gauge action where $F(e) = 1$ for all edges
$e$. The two sinks are both KMS sinks in this case; with intervals
$I_{s_1} = \mathbb R$ and $I_{s_2} = \left]\frac{\log 2}{2},
  \infty\right[$. Of the components it is only $C_2$ and $C_4$ that are KMS
  components, both of positive type and with $\beta_{C_2} =
  \beta_{C_4} = \frac{\log 2}{2}$. There are three extremal $\beta$-KMS
  states when $\beta = \frac{\log 2}{2}$, coming from $s_1,C_2$ and
  $C_4$, one when $\beta < \frac{\log 2}{2}$, coming from $s_1$, and
  two when $\beta > \frac{\log 2}{2}$, coming from $s_1$ and
  $s_2$. This 'KMS spectrum' away from $0$ can be
  described by the following figure.

\smallskip

\usetikzlibrary{arrows}
\begin{tikzpicture}
%\draw[latex-, very thick] (-7,0) -- (7,0);
%\draw[-latex, very thick] (-7,0) -- (7,0) node[below] {$\mathbb{R}$};
%\foreach \x in  {-4.16,0}
%\draw[shift={(\x,0)},color=black] (0pt,3pt) -- (0pt,-3pt);
%\draw[shift={(0,0)},color=black] (0pt,0pt) -- (0pt,-3pt) node[below] 
%{$0$};
%\draw[shift={(-4.16,0)},color=black] (0pt,0pt) -- (0pt,-3pt) node[below] 
%{$-\ln(4)$};
\draw[latex-, very thick] (-7,4) -- (7,4) ;
\draw[-latex, very thick] (-7,4) -- (7,4) node[below] {$\mathbb{R}$};
%\foreach \x in  {-2.08,0, 1.04}
%\draw[shift={(\x,4)},color=black] (0pt,3pt) -- (0pt,-3pt);
%\draw[shift={(0,4)},color=black] (0pt,0pt) -- (0pt,-3pt) node[below] 
%{$0$};
%\draw[shift={(-2.08,4)},color=black] (0pt,0pt) -- (0pt,-3pt) node[below] 
%{$-\ln(2)$};
\draw[shift={(1.04,4)},color=black] (0pt,0pt) -- (0pt,-3pt) node[below] 
{$\log(2)/2$};

\node at (-6.5,7.5) {$$};
\node at (7.3,4.5) {$C_{4}$};
\draw[fill] (1.04,4.5) circle (0.8mm);
\node at (7.3,5) {$C_{3}$};
\node at (7.3,5.5) {$C_{2}$};
\draw[fill] (1.04,5.5) circle (0.8mm);
\node at (7.3,6) {$C_{1}$};
\node at (7.3,6.5) {$s_{2}$};
\draw[o-] (0.94,6.5) -- (7,6.5);
\draw[-latex] (1.14,6.5) -- (7,6.5);
\node at (7.3,7) {$s_{1}$};
\draw[latex-] (-7,7) -- (7,7);
\draw[-latex] (-7,7) -- (7,7);
\end{tikzpicture}

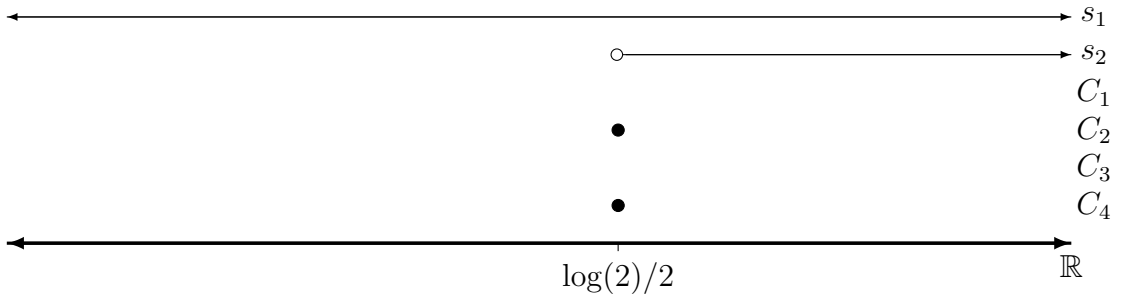
\begin{center} \emph{KMS spectrum ($\beta \neq 0$) for the gauge action on $C^*(G)$.}
\end{center}
\smallskip

To define a different generalized gauge action, let $E$ be the set of
edges in $G$, and set $F_1(e) = 1$ when $e \in E \backslash \{a,b,c\}$
while $F_1(a) = F_1(b) = -2$ and $F_1(c) = 0$. If we describe the
KMS-spectrum for the action $\alpha^{F_1}$ by a diagram as was done for
the gauge action, the picture becomes the following. The red line
describes the contribution from the circular KMS component $C_3$ and hence
each point on it
represents a family of extremal KMS states parametrized by a circle.

\smallskip

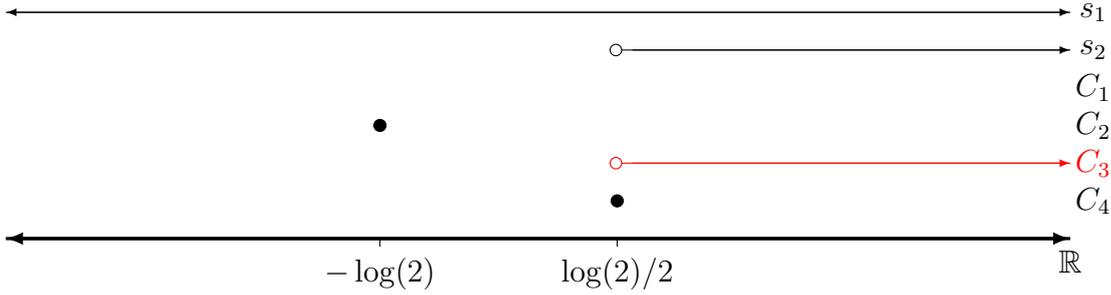
\usetikzlibrary{arrows}
\begin{tikzpicture}
%\draw[latex-, very thick] (-7,0) -- (7,0);
%\draw[-latex, very thick] (-7,0) -- (7,0) node[below] {$\mathbb{R}$};
%\foreach \x in  {-4.16,0}
%\draw[shift={(\x,0)},color=black] (0pt,3pt) -- (0pt,-3pt);
%\draw[shift={(0,0)},color=black] (0pt,0pt) -- (0pt,-3pt) node[below] 
%{$0$};
%\draw[shift={(-4.16,0)},color=black] (0pt,0pt) -- (0pt,-3pt) node[below] 
%{$-\ln(4)$};
\draw[latex-, very thick] (-7,0) -- (7,0) ;
\draw[-latex, very thick] (-7,0) -- (7,0) node[below] {$\mathbb{R}$};
%\foreach \x in  {-2.08,0, 1.04}
%\draw[shift={(\x,4)},color=black] (0pt,3pt) -- (0pt,-3pt);
%\draw[shift={(0,4)},color=black] (0pt,0pt) -- (0pt,-3pt) node[below] 
%{$0$};
\draw[shift={(-2.08,0)},color=black] (0pt,0pt) -- (0pt,-3pt) node[below] 
{$-\log(2)$};
\draw[shift={(1.04,0)},color=black] (0pt,0pt) -- (0pt,-3pt) node[below] 
{$\log(2)/2$};

%\node at (-6.5,3.5) {$F_{1}$};
\draw[latex-] (-7,3) -- (7,3);
\draw[-latex] (-7,3) -- (7,3);
\node at (7.3, 3) {$s_{1}$};
\draw[o-] (0.94,2.5) -- (7,2.5);
\draw[-latex] (1.24,2.5) -- (7,2.5);
\node at (7.3, 2.5) {$s_{2}$};
\node at (7.3,2) {$C_{1}$};
\node at (7.3,1.5) {$C_{2}$};
\draw[fill] (-2.08,1.5) circle (0.8mm);
\node[red] at (7.3,1) {$C_{3}$};
\draw[o-, red] (0.94,1) -- (7,1);
\draw[-latex, red] (1.24,1) -- (7,1);
\node at (7.3,0.5) {$C_{4}$};
\draw[fill] (1.04,0.5) circle (0.8mm);

\end{tikzpicture}

\begin{center} \emph{KMS spectrum ($\beta \neq 0$) for the generalized gauge action
  $\alpha^{F_1}$ on $C^*(G)$.}
\end{center}
\smallskip

Finally we consider $F_2$ defined such that $F_2(e) = 1$ when $e \in
E\backslash \{a,d\}$, $F_2(a) = -1$ and $F_2(d) = -\frac{3}{2}$. For
the generalized gauge action $\alpha^{F_2}$ we find the following  KMS
spectrum.

\smallskip

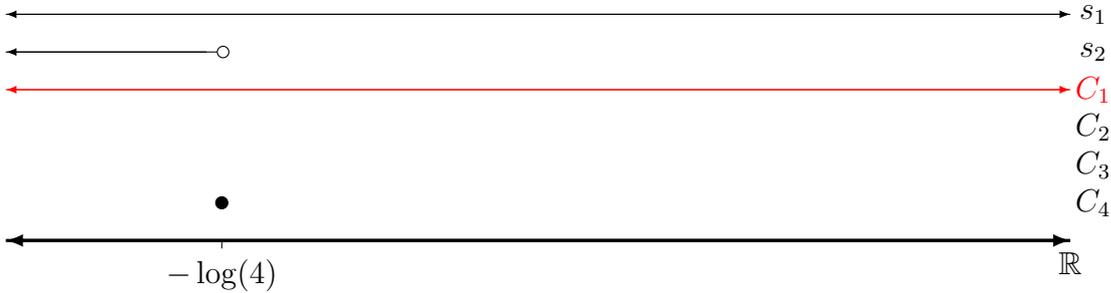
\usetikzlibrary{arrows}
\begin{tikzpicture}
\draw[latex-, very thick] (-7,-4) -- (7,-4);
\draw[-latex, very thick] (-7,-4) -- (7,-4) node[below] {$\mathbb{R}$};
%\foreach \x in  {-4.16,-4}
%\draw[shift={(\x,0)},color=black] (0pt,3pt) -- (0pt,-3pt);
%\draw[shift={(0,0)},color=black] (0pt,0pt) -- (0pt,-3pt) node[below] 
%{$0$};
\draw[shift={(-4.16,-4)},color=black] (0pt,0pt) -- (0pt,-3pt) node[below] 
{$-\log(4)$};
%\draw[latex-, very thick] (-7,4) -- (7,4) ;
%\draw[-latex, very thick] (-7,4) -- (7,4) node[below] {$\mathbb{R}$};
%\foreach \x in  {-2.08,0, 1.04}
%\draw[shift={(\x,4)},color=black] (0pt,3pt) -- (0pt,-3pt);
%\draw[shift={(0,4)},color=black] (0pt,0pt) -- (0pt,-3pt) node[below] 
%{$0$};
%\draw[shift={(-2.08,4)},color=black] (0pt,0pt) -- (0pt,-3pt) node[below] 
%{$-\ln(2)$};
%\draw[shift={(1.04,4)},color=black] (0pt,0pt) -- (0pt,-3pt) node[below] 
%{$\ln(2)/2$};

%\node at (-6.5,-0.5) {$F_{2}$};
\draw[latex-] (-7,-1) -- (7,-1);
\draw[-latex] (-7,-1) -- (7,-1) ;
\node at (7.3,-1) {$s_{1}$};
\draw[latex-] (-7,-1.5) -- (-4.36,-1.5);
\draw[-o] (-7,-1.5) -- (-4.06,-1.5);
\node at (7.3,-1.5) {$s_{2}$};
\draw[latex-, red] (-7,-2) -- (7,-2);
\draw[-latex, red] (-7,-2) -- (7,-2);
\node[red] at (7.3,-2) {$C_{1}$};
\node at (7.3,-2.5) {$C_{2}$};
\node at (7.3,-3) {$C_{3}$};
\node at (7.3,-3.5) {$C_{4}$};
\draw[fill] (-4.16,-3.5) circle (0.8mm);

\end{tikzpicture}

\begin{center} \emph{KMS spectrum ($\beta \neq 0$) for the generalized gauge action
  $\alpha^{F_2}$ on $C^*(G)$.}
\end{center}

\smallskip

The structure of the ground states vary also for the three
actions. For the gauge action there are two extremal ground states coming
from the sinks, while for the actions $\alpha^{F_1}$ and
$\alpha^{F_2}$ there are infinitely
many. Concerning $\alpha^{F_1}$ the sinks still contribute two, but the infinite path
$c^{\infty}$ has minimal $F_1$-weight and contributes a family of
extremal ground states naturally parametrized by a circle. The sink $s_1$ is the only sink which gives rise to an extremal
ground state for the
action $\alpha^{F_2}$, but now the loop of period 2 beginning with the edge $a$ is an element of
$\text{Min}(F_2,G)$ and gives rise to a family of
extremal ground states naturally parametrized by a circle.

The $0$-KMS states are of course the same for all three actions. They
are the trace states on the algebra, and by using Theorem \ref{L11} we
see that they can be identified with the trace states on $M_2(\mathbb
C) \oplus M_3(C(\mathbb T))$, where the sink $s_1$ is responsible for
the first summand and the component $C_1$ for the second.

\end{document}